\documentclass[a4paper,12pt]{article}
\usepackage{amsmath}
\usepackage{amssymb}
\usepackage{amsthm}
\usepackage{mathrsfs}
\usepackage{enumerate}
\theoremstyle{plain}
\newtheorem{theorem}{Theorem}[section]
\newtheorem{lemma}[theorem]{Lemma}
\newtheorem{proposition}[theorem]{Proposition}
\theoremstyle{remark}
\newtheorem{remark}{Remark}[section]

\newcommand{\Oun}{\mathscr{O}(1)}

\newcommand{\dvdx}{\,dv\,dx}
\newcommand{\dsdvdx}{\,ds\,dv\,dx}
\newcommand{\uper}{u_{\mathrm{per}}}
\newcommand{\lunp}{L^{1}_{\mathrm{per}}(ds\,dx)}
\newcommand{\gen}{\mathcal{L}}
\newcommand{\lunpx}{L^{1}_{\mathrm{per}}(dx)}
\newcommand{\lunpl}{\mathscr{B}}

\begin{document}
\title{Asymptotic velocity of one dimensional diffusions  with periodic drift}
\bigskip

\author{P. Collet\\
Centre de Physique Th\'eorique\\
CNRS UMR 7644\\
Ecole Polytechnique\\
F-91128 Palaiseau Cedex (France)\\
e-mail: collet@cpht.polytechnique.fr
\and S.Mart\'\i nez\\
Universidad de Chile,\\
 Facultad de Ciencias F\'{\i}sicas y Matem\'aticas,\\
Departamento de Ingenier\'{\i}a Matem\'atica \\
and Centro de Modelamiento Matem\'atico\\
UMR 2071 UCHILE-CNRS, \\
Casilla 170-3 Correo 3, Santiago, Chile.\\
e-mail: smartine@dim.uchile.cl
}

\maketitle

\begin{abstract}
We consider the asymptotic behaviour of the solution of one
dimensional stochastic differential equations and Lagevin equations in
periodic backgrounds with zero average. We prove that in several such
models, there is genrically a non vanishing asymptotic velocity,
despite of the fact that the average of the background is zero.
\end{abstract}

\section{Introduction.} 
We consider the one dimensional diffusion problem
\begin{equation}\label{modele1}
\partial_{t}u=\partial_{x}\left(\frac{1}{2}\partial_{x}u+b(t,x)u\right)
\end{equation}
where $b$ is a regular function periodic of period $T$ in the time
variable $t\in\mathbf{R}^{+}$ and of period $L$ in the space variable
$x\in\mathbf{R}$.  This equation and related equations discussed below
appear in many natural questions like molecular motors, population
dynamics, pulsed dielectrophoresis, etc.  See for example \cite{hhl},
\cite{han}, \cite{moreau} \cite{reimann}, \cite{quian} and references
therein.

We assume that the initial condition $u(0,x)$ is non negative and of
integral one with $|x|u(0,x)$ integrable, and denote by $u(x,t)$ the
solution at time $t$.  Note that the integral of $u$ with respect to
$x$ is invariant by time evolution, the integral of $xu(x,t)$ is
finite.

One of the striking phenomenon is that even if the drift has zero
average, this system may show a non zero average speed.  There are
many results on homogenization theory which can be applied to equation
(\ref{modele1}), see for example \cite{garnier}, \cite{molchanov}, and
references therein.

These results say that the large time asymptotic is given by the
solution of a homogenized problem. It remains however to understand if
this homogenized problem leads to a non zero asymptotic averaged
velocity. For this purpose we will consider the quantity
\begin{equation}\label{ib}
I(b)=\lim_{t\to\infty}\frac{1}{t}\int x u(t,x) dx
\end{equation}
which describes the asymptotic average displacement of the particle
per unit time (provided the limit exists).

Our first theorem states that the average asymptotic velocity is
typically non zero.

\begin{theorem}\label{principal1}
The quantity  $I(b)$ is independent of the initial condition, 
and the set of $b\in C^{1}$ with space average and time average equal to
zero where $I(b)\neq0$ is open and dense. 
\end{theorem}

\begin{remark}
By assuming that the space average (which may depend on time) and the
time average (which may depend on space) are both zero we restrict the
problem to a smaller set of possible drifts. One can prove a similar
result with weaker constraints. Note also that it is well known that if
$b$ does not depend on time, then $I(b)=0$ (see for example \cite{reimann}). 
\end{remark}

\begin{remark}
The theorem can be extended in various directions, for example by using
different topologies on $b$, by including a non-constant periodic
diffusion coefficient, or by considering almost periodic time dependence. 
\end{remark}

Another common model for molecular motor is the two states  model which
describes the time evolution of two non negative function $\rho_{1}$ and
$\rho_{2}$. In this model, the ``molecule'' can be in two states: $1$
or $2$ which have different interaction  with the landscape described
by the drift.
We denote by  $\rho_{1}(t,x)$  the probability to find
the molecule at site $x$ ate time $t$ in state  $1$, and
similarly for $\rho_{2}$. We refer to \cite{reimann} for more details.
The evolution equations are given by
\begin{equation}\label{modele2}
\begin{array}{ll}
\partial_{t}\rho_{1}=
&\partial_{x}\big(D\partial_{x}\rho_{1}+b_{1}(x)\rho_{1}\big)
-\nu_{1}\rho_{1}+\nu_{2}\rho_{2}\cr
\partial_{t}\rho_{2}=
&\partial_{x}\big(D\partial_{x}\rho_{2}+b_{2}(x)\rho_{2}\big)
+\nu_{1}\rho_{1}-\nu_{2}\rho_{2}\cr
\end{array}
\end{equation}
where $D$, $\nu_{1}$ and $\nu_{2}$ are positive constants,  $b_{1}$
and $b_{2}$ are $C^{1}$ periodic functions of $x$ of period $L$ the
last two with average zero.

The asymptotic average displacement per unit time 
of the particle is now defined by
$$
I\big(\nu_{1},\nu_{2},b_{1},b_{2}\big)=
\lim_{t\to\infty}\frac{1}{t}
\int x\,\big(\rho_{1}(t,x)+\rho_{2}(t,x)\big)\,dx\;.
$$
We have the equivalent of Theorem \ref{principal1}. As before, we
assume that $|x|(\rho_{1}+\rho_{2})$ is integrable.

\begin{theorem}\label{principal2} For any constants  $\nu_{1}>0$ and
  $\nu_{2}>0$, $I(\nu_{1},\nu_{2},b_{1},b_{2})$ is independent of the
 initial condition, and the set of   $b_{1}$ and $b_{2}$
 $\in C^{1}$  with space average equal to
zero where $I(\nu_{1},\nu_{2},b_{1},b_{2})\neq0$ is open and dense.
\end{theorem}

Another model of a particle interacting with a thermal bath is given
by the Langevin equation
\begin{equation}\label{langevin}
\begin{array}{ll}
dx=&\!\!\!v\,dt\\
dv=&\!\!\!(-\gamma\,v+F(t,x)/m)\,dt+\sigma \,dW_{t}
\end{array}
\end{equation} 
where $m$ is the mass of the particle, $\gamma>0$ the friction
coefficient, $F(t,x)$ the force, $W_{t}$ the Brownian motion and
$\sigma=\sqrt{2D}$ where $D$ is the diffusion coefficient.  We refer
to \cite{hhl} and \cite{moreau} for more details.

For the time evolution of the probability density $f(t,v,x)$ of the
position and velocity of the particle one gets the so called Kramers
equation
\begin{equation}\label{kramers}
\partial_{t}f=-v\,\partial_{x}f+\partial_{v}\big[(\gamma v-F(t,x)/m)f\big]
+\frac{D}{2}\partial_{v}^{2}f\;.
\end{equation}

We refer to \cite{hhl} and references therein for more details on
these equations. By changing scales, we can assume that $m=1$ and
$D=1$ and we will only consider this situation below. Moreover we will
assume as before that $F(t,x)$ is periodic of period $T$ in time, $L$
in space and with zero average in space and time.  We can now define
the average asymptotic displacement per unit time by
\begin{equation}\label{vit3}
I(\gamma,F)=\lim_{t\to\infty}\frac{1}{t}\int_{0}^{t}d\tau
\int\!\!\!\int v\,f(\tau,v,x)\,dv\,dx\;.
\end{equation}

As for the previous models the average asymptotic velocity is
typically non zero. As usual, we denote by $H^{1}(\dvdx)$ the Sobolev
space of square integrable functions of $x$ and $v$ with square
integrable gradient.

\begin{theorem}\label{principal3}
For $\gamma>0$, $I(\gamma,F)$ is independent of the initial condition,
the set of $F\in C^{1}$ with space average and time
average equal to zero where $I(\gamma,F)\neq0$ 
is open and dense.
\end{theorem}

One can also consider a situation where the particle can be in 
two states which interact differently with the landscape. This leads
to the following system of  Kramers equation.

\begin{equation}\label{kramers2}
\begin{array}{ll}
\partial_{t}f_{1}=&\frac{1}{2}\partial_{v}^{2}f_{1}-v\partial_{x}f_{1}
+\partial_{v}\big[(\gamma v-F_{1}(x))f_{1}\big]
-\nu_{1}f_{1}+\nu_{2}f_{2}\\
\partial_{t}f_{2}=&\frac{1}{2}\partial_{v}^{2}f_{2}-v\partial_{x}f_{2}
+\partial_{v}\big[(\gamma v-F_{2}(x))f_{2}\big]
+\nu_{1}f_{1}-\nu_{2}f_{2}\;.
\end{array}
\end{equation}

In this equation, $F_{1}$ and $F_{2}$ are two periodic functions
representing the different interaction forces between the two states
of the particle and the substrate. The positive constants $\nu_{1}$
and $\nu_{2}$ are the transition rates between the two states. The non
negative functions $f_{1}$ and $f_{2}$ are the probability densities
of being in state one and two respectively. The total probability
density of the particle is the function $f_{1}+f_{2}$ which is
normalised to one. The asymptotic displacement per unit time for this
model is given by

\begin{equation}\label{vit4}
I(\gamma,F_{1},F_{2},\nu_{1},\nu_{2})
=\lim_{t\to\infty}\frac{1}{t}\int_{0}^{t}ds
\int\!\!\!\int v\,\big(f_{1}(s,v,x)+f_{2}(s,v,x)\big)\,dv\,dx\;,
\end{equation}
and we will prove the following result

\begin{theorem}\label{principal4}
For $\gamma>0$, $\nu_{1}>0$ and $\nu_{2}>0$,
$I(\gamma,F_{1},F_{2},\nu_{1},\nu_{2})$ is independent of the initial
condition, and the set of $F_{1}$ and $F_{2}\in C^{1}$ with space
average equal to zero where
$I(\gamma,F_{1},F_{2},\nu_{1},\nu_{2})\neq0$ is open and dense.
\end{theorem}

\section{Elimination  of the spatial average.}

Before we start with the proof of the Theorems, we first show that the
result does not depend on the spatial average of the drift $b$.

\begin{proposition}\label{moyspa}
Assume $b$ has space time average zero, namely
$$
\frac{1}{TL}\int_{0}^{T}\int_{0}^{L} b(t,x) \,dt\,dx=0\;.
$$
 Then the drift  $\tilde b$ given by
$$
\tilde b(t,x)=b(t,x+a(t))-\frac{1}{L}\int_{0}^{L}b(t,y)dy
$$
where
$$
a(t)=-\frac{1}{L}\int_{0}^{t}ds\int_{0}^{L}b(s,y)dy\;.
$$
is periodic of period $T$ in time and of period $L$ in x. This drift
has zero space average and 
leads to the same asymptotic displacement per unit time.
\end{proposition}

\begin{proof}
Note first that since the space time average of $b$ is zero, the
function $a$ is periodic of period $T$. Let $u$ be a solution of
(\ref{modele1}), and define the function
$$
v(t,x)=u(t,x+a(t))\;.
$$
An easy computation leads to
$$
\partial_{t}v=\partial_{x}\left(\frac{1}{2}\partial_{x}v+\tilde
b(t,x)v\right)\;. 
$$
Since $a(t)$ is periodic and bounded we have by a simple change of
variable
$$
\lim_{t\to\infty}\frac{1}{t}\int x\,u(t,x)dx=
\lim_{t\to\infty}\frac{1}{t}\int x\,v(t,x)dx\;.
$$
\end{proof}

\section{Proof of Theorems \ref{principal1}}

We start by giving a more convenient expression for the asymptotic
velocity $I(b)$.  Using (\ref{ib}) and (\ref{modele1}), and
integrating by parts we get
$$
I(b)=\lim_{t\to\infty}\frac{1}{t}\int_{0}^{t}ds \int x\, 
\partial_{s}u(s,x) dx
=-\lim_{t\to\infty}\frac{1}{t}\int_{0}^{t}ds \int b(s,x)\,u(s,x)dx\;.
$$
Since $b$ is periodic in $x$, of period $L$, we can write
$$
\int b(s,x)\,u(s,x)dx=\int_{0}^{L}  b(s,x)\,\uper(s,x)dx
$$
where
$$
\uper(s,x)=\sum_{n}u(s,x+nL)
$$
is a periodic function of $x$ of period $L$. Note that since $b$ is
periodic of period $L$, $\uper$ satisfies also equation
(\ref{modele1}).
We now have
\begin{equation}\label{vas}
I(b)=-\lim_{t\to\infty}\frac{1}{t}\int_{0}^{t}ds \int_{0}^{L}
 b(s,x)\,\uper(s,x)dx\;.
\end{equation}
Since the system is non autonomous, although periodic in time, we can
only expect that when $t$ tends to infinity, the function $\uper(t,x)$
tends to a periodic function $w_{b}(t,x)$ of $t$ and $x$.  Let $w_{b}$
be the solution of equation (\ref{modele1}) periodic in space and time
and with an integral (over $[0,L]$) equal to one.  It can be expected
(see \cite{reimann}, \cite{garnier}, \cite{molchanov}) that the
asymptotic average displacement is given by
\begin{equation}\label{connu}
I(b)=-\frac{1}{T}\int_{0}^{T}\int_{0}^{L}b(t,x)w_{b}(t,x)dt dx\;.
\end{equation}

In order to give a rigorous proof of existence of the function
$w_{b}(t,x)$ and of the above relation, we introduce a new time. We
consider the operator $\gen$ given by
$$
\gen\,w=
-\partial_{s}w+\partial_{x}\left(\frac{1}{2}\partial_{x}w+b(s,x)w\right)\;,
$$ 
acting in a suitable domain dense in the space $\lunp$ of integrable
functions which are periodic in $s$ and $x$ of periods $T$ and $L$
respectively .  This operator is the generator of the diffusion on the
two dimensional torus ($[0,T]\times [0,L]$ with the suitable
identifications) associated to the stochastic differential equation
\begin{equation}\label{diff1}
\left\{
\begin{array}{ll}
ds&=dt\\
dx&=-b(s,x)\,dt+dW_{t}
\end{array}
\right.
\end{equation}
where $W_{t}$ is the standard Brownian motion (see \cite{iw}).
We can now establish the following result.

\begin{proposition}\label{deuxtemps}
The diffusion (\ref{diff1}) has a unique invariant probability measure
with density $w_{b}(s,x)$. This function is strictly positive. It is
periodic of period $T$ in $s$ and of period $L$ in $x$ and satisfies
equation (\ref{modele1}), and it is the only such solution. The semi
group with generator $\gen$ associated to the diffusion (\ref{diff1})
is compact and strongly continuous. The peripheral spectrum of its
generator is composed of the simple eigenvalue zero (with eigenvector
$w_{b}$). In particular, for any function $v\in\lunp$, we have in the
topology of $\lunp$
$$
\lim_{\tau\to\infty} e^{\tau \gen}v=w_{b} \int_{0}^{T}\int_{0}^{L}
v(s,x)\,ds\,dx\;. 
$$
\end{proposition}

This kind of results is well known, we refer to \cite{luc} for an
exposition and further references. 
We can now establish the relation between  (\ref{vas}), and (\ref{connu}). 

\begin{proposition}\label{asympt1}
Let $v_{0}\ge0$ be a periodic function of period $L$ in $x$ of
integral one. Denote by $v(t,x)$ the solution of (\ref{modele1}) which
is periodic of period $L$ in $x$ with initial condition $v_{0}$. Then
$$
\lim_{t\to\infty}\frac{1}{t}\int_{0}^{t}ds \int_{0}^{L}
 b(s,x)\,v(s,x)dx=
-\frac{1}{T}\int_{0}^{T}\int_{0}^{L}b(t,x)w_{b}(t,x)dt dx\;.
$$
\end{proposition}

\begin{proof}

In order to apply Proposition (\ref{deuxtemps}), we consider the operator
$\gen_{0}$ given by
$$
\gen_{0}=\partial_{x}\left(\frac{1}{2}\partial_{x}u+b(s,x)u\right)\;,
$$
and observe that if $w\in \lunp$, we have
\begin{equation}\label{letruc}
\left(e^{\tau \gen}w\right)(s,x)=\left(e^{\tau
  \gen_{0}}w(s-\tau,\,\cdot\,)\right)(x)\;,
\end{equation}
and in particular for any integer $n$, we get
$$
\left(e^{nT\gen}w\right)(s,x)=\left(e^{nT
  \gen_{0}}w(s,\,\cdot\,)\right)(x)\;,
$$
since $w$ is of period $T$ in $s$.

We now take for $w$ the function $w(s,x)=v(s,x)$ for $0\le
s<T$. Although this $w$ may have a jump at $s=T$, we can consider it
as a function in $\lunp$.  We observe  that if $W(\tau,s,x)$ is a
solution of
$$
\partial_{\tau}W=
-\partial_{s}W+\partial_{x}\left(\frac{1}{2}\partial_{x}W+b(s,x)W\right)\;,
$$
then for each fixed $s_{0}$, the function
$h_{s_{0}}(\tau,x)=W(\tau,s_{0}+\tau,x)$ is a solution of
$$
\partial_{\tau}h_{s_{0}}=
\partial_{x}\left(\frac{1}{2}\partial_{x}h_{s_{0}}
+b(s_{0}+\tau,x)h_{s_{0}}\right)\;.
$$
Therefore, by the uniqueness of the
solution of (\ref{modele1}), we have  for any $t\ge 0$ (taking
$s_{0}=t-T[t/T]$)
$$
v(t,x)=\left( e^{[t/L]T\gen}w\right)(t-[t/L]T,x)\;.
$$
The proposition follows by applying Proposition \ref{deuxtemps}.

\end{proof}
The following proposition is the other main step in the proof of Theorem 
\ref{principal1}.

\begin{proposition}\label{analytique}
The function $b\mapsto I(b)$ is (real) analytic in the Banach space $C^{1}$.
\end{proposition}

By this we mean (see \cite{mujica})
 that the function is $C^{\infty}$, and around
any point $b\in C^{1}$  there is a small ball where the Taylor series
converges to the function.

\begin{proof}
We will establish that the map $b\mapsto w_{b}$ is real analytic in
$\lunp$. For this purpose, we first establish that the operator $A$
defined by
$$
Av=\partial_{x} (bv)=b\,\partial_{x}v+\partial_{x}b\, v 
$$ is relatively
bounded with respect to
$$
\tilde\gen=-\partial_{s}+\frac{1}{2}\partial_{x}^{2}\;,
$$
and with relative bound zero (see \cite{kato} for the definition).
 This is obvious for the operator of
multiplication by $\partial_{x}b$ which is bounded, and since $b$ is
bounded it is enough to derive the result for the operator
$\partial_{x}$. We will show that there is a constant $C>0$ such that
for any $\lambda>0$,
$\big\|\partial_{x}R_{\lambda}\big\|_{\lunp}<C\lambda^{-1/2}$, where
$R_{\lambda}$ is the resolvent of $\tilde\gen$.
In other words, we will show that for  any $\lambda>0$
\begin{equation}\label{relat}
\left\|\partial_{x}\int_{0}^{\infty}e^{-\lambda\tau}e^{\tau \tilde\gen}d\tau
\right\|_{\lunp}<\frac{C}{\sqrt \lambda}\;.
\end{equation}

Analogously to formula (\ref{letruc}) we have for any $w\in\lunp$
$$
\partial_{x}
\left(e^{\tau \gen'}w\right)(s,x)=
\int_{0}^{L} \partial_{x} g_{\tau}(x,y)\,w(s-\tau,y)\,dy\;,
$$
where $g_{\tau}(x,y)$ is the heat kernel on the circle of length
$L$. We now observe that if $n$ is an integer with
$|n|\ge 2$, we have (since $x\in[0,L]$) 
$$
\sup_{y\in[0,L]}\int_{0}^{L} 
\frac{\big|x-y-nL\big|}{\tau^{3/2}}\;
 e^{-(x-y-nL)^{2}/(2\tau)}dx\le\Oun\frac{e^{-n^{2}/(4\tau)}}{\tau}\;.
$$
From the explicit expression
$$
g_{\tau}(x,y)=\sum_{n}\frac{1}{\sqrt{2\,\pi\,\tau}}\; e^{-(x-y-nL)^{2}/(2\tau)}
$$
it follows easily that 
$$
\sup_{y\in[0,L]}\int_{0}^{L} \Big| \partial_{x}
g_{\tau}(x,y)\Big|dx
\le \sum_{|n|\le 1} \sup_{y\in[0,L]}\int_{0}^{L} 
\frac{\big|x-y-nL\big|}{\tau^{3/2}\sqrt{2\pi}}\;
 e^{-(x-y-nL)^{2}/(2\tau)}dx 
$$
$$
+\Oun\sum_{|n|\ge 2}
\frac{e^{-n^{2}/(4\tau)}}{\tau}
\le\frac{\Oun}{\sqrt \tau}\;.
$$
Therefore, we get
$$
\left\|\partial_{x}
\left(e^{\tau \gen'}w\right)\right\|_{\lunp}
\le \frac{\Oun\|w\|_{\lunp}}{\sqrt \tau}\;.
$$
Multiplying by $e^{-\lambda \tau}$ and integrating over $\tau$ we get
the estimate (\ref{relat}) which implies immediately that $A$ is
relatively bounded with respect to $\gen'$ with relative bound zero.
Since the eigenvalue $0$ of $\gen$ is simple and isolated, the
proposition follows from analytic perturbation theory for holomorphic
families of type (A) (see \cite{kato}). 
\end{proof}

In order to prove Theorem \ref{principal1}, we now establish that $I$
is non trivial near the origin.

\begin{lemma}\label{pastrivial}
$DI_{0}=0$ and
$$
D^{2}I_{0}(B,B)=\sum_{p,q}\frac{p\,qL^{3}T^{2}}{p^{2}L^{4}
+\pi^{2}q^{4}T^{2}}\big|B_{p,q}\big|^{2}\;.
$$
\end{lemma}

\begin{proof}
The first statement is immediate from formula (\ref{connu}).
For $b=0$, we have obviously $w_{0}=1/L$.
For $b$ small we use Taylor expansion. We have for
$b=\epsilon B$
$$
w_{b}=w^{0}+\epsilon w^{1}+\mathscr{O}(\epsilon^{2})\;.
$$
As we just explained, $w^{0}=1/L$ and we have for $w^{1}$ the equation
$$
\partial_{t}w^{1}=\frac{1}{2}\partial_{x}^{2}w^{1}+\frac{1}{L}\partial_{x}B\;.
$$
Moreover, $w^{1}$ must have space time average zero. 
This equation can be solved using Fourier series in time and
space. Namely if
$$
w^{1}(t,x)=\sum_{p,q}e^{2\pi ip t/T}e^{2\pi iq x/L}w^{1}_{p,q}\;,
$$
we obtain the equation
$$
\frac{2\pi ip}{T}w^{1}_{p,q}=-\frac{2\pi^{2}q^{2}}{L^{2}}w^{1}_{p,q}+
\frac{2\pi iq}{L^{2}}B_{p,q}\;,
$$
or in other words for any $(p,q)\neq (0,0)$
$$
w^{1}_{p,q}=\frac{\pi iq/L^{2}}{\pi ip/T+\pi^{2}q^{2}/L^{2}}B_{p,q}\;.
$$
Note in particular that the denominator does not vanish except for
$p=q=0$. 

Using the Plancherel formula we can now estimate $I(\epsilon B)$. We
have
$$
I(\epsilon B)=\frac{\epsilon^{2}}{T}\int B(t,x) w^{1}(t,x) dt dx
+\Oun\epsilon^{3}\;. 
$$
Therefore 
$$
I(\epsilon B)=\epsilon^{2}\sum_{p,q}
\frac{q/L}{p/T-i\pi q^{2}/L^{2}}B_{p,q}\overline{B_{p,q}}
+\mathbf{O}(\epsilon^{3})\;.
$$
Since $B$ is real, we have $\overline{B_{p,q}}=B_{-p,-q}$ and this
can also be written
$$
I(\epsilon B)=\epsilon^{2}\sum_{p,q}\frac{p\,qL^{3}T^{2}}{p^{2}L^{4}
+\pi^{2}q^{4}T^{2}}\big|B_{p,q}\big|^{2}+\mathbf{O}(\epsilon^{3}) \;.
$$
This finishes the proof of the Lemma.
\end{proof}

To prove Theorem \ref{principal1}, we observe that since $I$ is
continuous, the subset of $C^{1}$ where it does not vanish is open. 
If this set is not dense, the zero set of $I$ contains a ball. However
since $I$ is real analytic and $C^{1}$ is pathwise connected we conclude
that in that case $I$ should vanish identically contradicting Lemma
\ref{pastrivial}. We refer to  \cite{mujica}
 for more properties of the zero set of 
analytic functions in Banach spaces.

\begin{remark}
We observe that $D^{2}I$ is a non definite quadratic form. This leaves 
the possibility of having non zero drifts $b$ (with space and time
average equal to zero)  satisfying $I(b)=0$. Let $b_{1}$ and $b_{2}$
be such that $I(b_{1})>0$ and $I(b_{2})<0$. Such $b_{1}$ and $b_{2}$
exist, one can for example take them of small enough norm and use Lemma 
\ref{pastrivial}. Moreover one can assume that $b_{2}\notin
\mathbf{R}b_{1}$. Otherwise, by the continuity of $I$, one can perturb 
 slightly $b_{2}$ such that this relation does not hold anymore, but $I(b_{2})$
is still negative. One now considers the function $\varphi(\alpha)=
I\big((1-\alpha)b_{1}+\alpha b_{2}\big)$. This function is continuous, it
satisfies $\varphi(1)>0$ and $\varphi(0)<0$, hence it should vanish at least
at one point $\alpha_{0}\in ]0,1[$. At this point we have
$b_{0}=(1-\alpha_{0})b_{1}+\alpha_{0} b_{2}\neq 0$ and $I(b_{0})=0$, and
$b_{0}$ is a non trivial periodic function with vanishing space average
and time average.
\end{remark}

\section{Proof of Theorem \ref{principal2}.}

In this section we discuss the model with two components
(\ref{modele2}).  As before, we arrive at  the formula 
$$
I\big(\nu_{1},\nu_{2},b_{1},b_{2}\big)=-
\lim_{t\to\infty}\frac{1}{t}\int_{0}^{t}
\int_{0}^{L} \big(b_{1}(s,x)\,\rho_{1}(s,x)+b_{2}(s,x)\,
\rho_{2}(s,x)\big)\,ds\,dx\;,
$$
where $\rho_{1}$ and $\rho_{2}$ are solutions of (\ref{modele2})
periodic in $x$ of period $L$.  We denote by $\lunpx$ the space of
integrable periodic functions  of period $L$ in $x$ with value in
$\mathbf{R}^{2}$. The norm is the sum of the $L^{1}$ norms of the
components.  

\begin{proposition}\label{et2}
The semi-group defined by (\ref{modele2}) 
is  compact  in $\lunpx$. It is positivity preserving, and its
peripheral spectrum is the simple eigenvalue one.
The corresponding eigenvector  can be chosen
 positive with dense support and normalised,
 it depends analytically on $b_{1}$,
and $b_{2}$. 
\end{proposition}

\begin{proof}

We introduce the three generators
$$
\gen\left(\begin{array}{l}
\rho_{1}\\
\rho_{2}\\
\end{array}\right)
=\left(\begin{array}{l}
\partial_{x}\big(D\partial_{x}\rho_{1}+b_{1}(x)\rho_{1}\big)
-\nu_{1}\rho_{1}+\nu_{2}\rho_{2}\\
\partial_{x}\big(D\partial_{x}\rho_{2}+b_{2}(x)\rho_{2}\big)
+\nu_{1}\rho_{1}-\nu_{2}\rho_{2}\\
\end{array}\right)\;,
$$

$$
\gen_{0}\left(\begin{array}{l}
\rho_{1}\\
\rho_{2}\\
\end{array}\right)
=\left(\begin{array}{l}
D\partial_{x}^{2}\rho_{1}\\
D\partial_{x}^{2}\rho_{2}\\
\end{array}\right)
\quad\mathrm{and} \quad
\gen_{1}\left(\begin{array}{l}
\rho_{1}\\
\rho_{2}\\
\end{array}\right)
=\left(\begin{array}{l}
\partial_{x}\big(D\partial_{x}\rho_{1}+b_{1}(x)\rho_{1}\big)\\
\partial_{x}\big(D\partial_{x}\rho_{2}+b_{2}(x)\rho_{2}\big)\\
\end{array}\right)
\;.
$$
This operator $\gen_{0}$ is the infinitesimal generator  of
a strongly continuous bounded and compact
semi-group in $\lunpx$. It is easy to verify that the operator 
$A=\gen-\gen_{0}$ is $\gen_{0}$ relatively compact and  $\gen_{0}$
relatively bounded with relative bound zero (see \cite{kato}).  Therefore 
$\gen$ is also  the infinitesimal generator 
 a strongly continuous  and compact
semi-group in $\lunpx$, and similarly for $\gen_{1}$. 
 
The semi-group $e^{t\gen}$ is positivity improving (see \cite{davies}). 
Indeed, let $M$ be 
the matrix 
$$
M=\left(\begin{array}{cc}
-\nu_{1}&\nu_{2}\\
\nu_{1}&-\nu_{2}\\
\end{array}\right)\;.
$$
It is easy to verify for example by direct computation, 
that the matrix $e^{t\,M}$ has strictly positive
entries for any $t>0$. Moreover, for any $t\ge 0$, we have 
$e^{t\,M}\ge e^{-t(\nu_{1}+\nu_{2})}\mathrm{Id}$, 
where the inequality holds for each entry.
It immediately follows from the Trotter product formula (see
\cite{davies} ) that for each $x$ and $y$ in the circle, and
any $t>0$, we have
$$
e^{t\gen}(x,y)\ge e^{-t(\nu_{1}+\nu_{2})}e^{t\gen_{1}}(x,y)\;,
$$
again in the sense that the inequality holds between all the
entries.  Since for each $x$, $y$ and $t>0$ the diagonal elements of $
e^{t\gen_{1}}(x,y)$ are strictly positive (see for example
\cite{luc}),
 we conclude that the matrix valued kernel $e^{t\gen}(x,y)$  has non
 negative entries and strictly positive entries on the diagonal. 

Since the sum of the integrals of the two components of an element of $\lunpx$
is preserved by the semi-group $e^{t\gen}$, it follows that this semi
group has norm one in $\lunpx$. 
It then follows by classical arguments
that $0$ is a simple isolated eigenvalue of the generator and there is
no other eigenvalue with vanishing real part.

The analyticity follows from the uniqueness and simplicity of the
eigenvalue $0$  as in the proof of proposition (\ref{analytique}).

\end{proof}

We denote by  $w_{1}$ and $w_{2}$ 
the two (non negative) components of the
stationary solution of the system (\ref{modele2}) which are periodic
of period $L$ and normalised by 
$$
\int_{0}^{L}\big[w_{1}(x)+w_{2}(x)\big]dx=1\;.
$$
Note that  $w_{1}$ and $w_{2}$  depend on the constants 
$\nu_{1}$, $\nu_{2}$, and the functions $b_{1}$, and $b_{2}$.

It follows immediately from Proposition  \ref{et2} 
that the  average asymptotic velocity
is  given by

\begin{equation}\label{connu2}
I(\nu_{1}, \nu_{2}, b_{1},b_{2})=\int_{0}^{L}\big[b_{1}(x)w_{1}(x)+
b_{2}(x)w_{2}(x)\big]
dx\;.
\end{equation}

Since the function $I(\nu_{1},
\nu_{2},b_{1},b_{2})$ is analytic in $\big(b_{1},b_{2}\big)$,
 to prove that it is non trivial, we
look at the successive differentials at the origin.

\begin{proposition}
$DI_{0}=0$ and for any $(b_{1},b_{2})$,
$D^{2}I_{0}\big((b_{1},b_{2}),(b_{1},b_{2})\big)=0$. 
\end{proposition}

\begin{proof}
The first result is trivial. For the second result one uses perturbation
theory as before in the Fourier decomposition. We get with
$\sigma=4\pi^{2}D/L^{2}$ 
$$
\begin{pmatrix}
-\sigma n^{2}-\nu_{1}& \nu_{2}\\
\nu_{1} & -\sigma n^{2}-\nu_{2}\\
\end{pmatrix}
\begin{pmatrix}w_{1}^{1}(n)\\
w_{1}^{2}(n)\\
\end{pmatrix}=-\frac{2\pi\,i\,n}{L(\nu_{1}+\nu_{2})}
\begin{pmatrix}
\nu_{2}b_{1}(n)\\
\nu_{1}b_{2}(n)\\
\end{pmatrix}
$$
Some easy computations using Plancherel identity lead to
$$
D^{2}I_{0}\big((b_{1},b_{2}),(b_{1},b_{2})\big)=-
\sum_{n} \frac{2\pi\,i\,n}{L(\nu_{1}+\nu_{2})
\big((\sigma n^{2}+\nu_{1})(\sigma n^{2}+\nu_{2})-\nu_{1}\nu_{2}\big)}
$$
$$
\times \Big[(\sigma n^{2}+\nu_{1})\nu_{2}\bar b_{1}(n)  b_{1}(n)+
(\sigma n^{2}+\nu_{2})\nu_{1}\bar b_{2}(n)  b_{2}(n)
$$
$$
+\nu_{1}\nu_{2} 
\big( \bar b_{2}(n)  b_{1}(n)+\bar b_{1}(n)  b_{2}(n)\big)\Big]=0
$$
since $b_{1}$ and $b_{2}$ are real ($\bar b_{1}(n)=b_{1}(-n)$).
\end{proof}

This result suggests to look at the third differential at the origin
which turns out to be a rather involved cubic expression. In order to
show that the function $I$ is non trivial, 
it is enough to find a particular pair
$\big(b_{1}, b_{2}\big)$ such that
$D^{3}I_{0}\big((b_{1},b_{2}),(b_{1},b_{2}),(b_{1},b_{2})\big)\neq0$. 
This was done using a symbolic manipulation program (Maxima). We found
that  for $L=2\pi$, $D=1$, $b_{1}(x)=\cos(2x)$ and $b_{2}(x)=\cos(x)$, one
gets
$$
D^{3}I_{0}\big((b_{1},b_{2}),(b_{1},b_{2}),(b_{1},b_{2})\big)=
-\frac{ \nu_{1} \nu_{2} (\nu_{2} - 2 \nu_{1} + 1)}{ 4 (\nu_{2} +
  \nu_{1}) (\nu_{2} + \nu_{1} + 1)^{2}  (\nu_{2} + \nu_{1} + 4)}\;. 
$$
Theorem \ref{principal2} follows as before. 

\subsection{Proof of Theorem \ref{principal3}}

As in the previous section, we can introduce the periodised function
(in $x$) 
$$
\breve f(t,v,x)=\sum_{n}f(t,v,x+nL)\;.
$$
This function is periodic of period $L$ and satisfies also equation 
(\ref{kramers}). We get also $I(\gamma,F)$ by replacing $f$ by $\breve
f(t,v,x)$ in equation (\ref{vit3}) and integrating only on one
period. From now on we will work with this  periodised function and
denote it by $f$ by abuse of notation.

We now introduce a stochastic differential equation on
$[0,T]\times [0,L]\times \mathrm{R}$ with periodic boundary conditions
in the first two variables $s$ and $x$. This differential equation is
given by
\begin{equation}\label{diff2}
\left\{\begin{array}{l}
ds=dt\\
dx=v\,dt\\
dv=-\gamma v dt+F(s,x)dt+dW_{t}\;.
\end{array}\right.
\end{equation}
To this diffusion is associated the infinitesimal generator $\gen$
 given by
$$
\gen w=-\partial_{s} w
-v\,\partial_{x}w+\partial_{v}\big[(\gamma v-F(t,x))w\big]
+\frac{1}{2}\partial_{v}^{2}w\;.
$$ 
We denote by $\lunpl$ the space $L^{2}(e^{\gamma v^{2}} ds\,\dvdx)$
of functions periodic in $s$ of period $T$ and periodic in $x$ of
period $L$.  Using an $L^{2}$ space instead of an $L^{1}$ space is
useful in proving analyticity.

We can now establish the following result.

\begin{proposition}\label{deuxtempsl}
The diffusion semi-group  defined by (\ref{diff2}) in $\lunpl$ is
compact and the kernel has dense support. It is mixing and 
 has a unique invariant probability measure
(in $\lunpl$)
with density $\tilde f(s,v,x)$. This function is strictly positive, satisfies
equation (\ref{kramers}), and it is the only such solution. 
In particular, for any
function $w\in\lunpl$, we have in the topology of $\lunpl$
$$
\lim_{\tau\to\infty} e^{\tau \gen}w=\tilde f \int_{0}^{T}\int_{0}^{L}\int
w(s,v,x)\,ds\dvdx\;.
$$
The function  $\tilde f$ is real analytic in $F\in C^{1}$. 
\end{proposition}

\begin{proof}

Instead of working in the space $\lunpl$ , we can work in the
space $L^{2}(ds\,\dvdx)$ by using the isomorphism given by the
multiplication by function $e^{\gamma v^{2}/2}$. In that space we
obtain the new generator $\gen'_{F}$ given by
$$
\gen'_{F} g=-\partial_{s} g
-v\,\partial_{x}g-F\partial_{v}g +F\gamma v g
-\frac{\gamma^{2}v^{2}}{2}g+\frac{\gamma}{2}g
+\frac{1}{2}\partial_{v}^{2}g\;.
$$ Using integration by parts, it is easy to verify that
\begin{equation}\label{coercif}
\Re \int \bar g\,\big(\gen'_{F}\,g\big)  ds\,\dvdx=
$$
$$
\int\left(\frac{\gamma}{2} |g|^{2}+\gamma  F\,v\,|g|^{2}
-\frac{\gamma^{2}}{2} v^{2} |g|^{2}
-\frac{1}{2}\int \big|\partial_{v}g\big|^{2}\right)\dsdvdx
$$
$$
\le 
\left(\frac{\gamma}{2}+\big\|F\big\|_{\infty}^{2}\right)\int 
|g|^{2}\dsdvdx
-\int\left(\frac{\gamma^{2}}{4} v^{2} |g|^{2}
+\frac{1}{2} \big|\partial_{v}g\big|^{2}\right)\dsdvdx\;.
\end{equation} 
We see immediately that $-\gen'_{F}$ is quasi accretive (see \cite{kato} for
the definition and properties of the semi-groups generated by these
operators).

Let $g_{t}=e^{t\gen'_{F}}g_{0}$, using several integrations by parts
one gets easily
$$
\partial_{t}\int\big|g_{t}\big|^{2}\dsdvdx=2\gamma 
\int F\,v\,\big|g_{t}\big|^{2}\dsdvdx-\gamma^{2}\int v^{2}\,
\big|g_{t}\big|^{2}\dsdvdx
$$
$$
+\gamma \int\big|g_{t}\big|^{2}\dsdvdx
-\int\big|\partial_{v}g_{t}\big|^{2}\dsdvdx
$$ 
$$
\le -\frac{\gamma^{2}}{2}\int v^{2}\,
\big|g_{t}\big|^{2}\dsdvdx +
\big(\gamma +4\gamma^{2}\big\|F\big\|_{\infty}^{2}\big)
\int\big|g_{t}\big|^{2}\dsdvdx\;.
$$
We obtain immediately
\begin{equation}\label{borne0}
\big\|g_{t}\big\|\le e^{\big(\gamma/2
  +2\gamma^{2}\|F\|_{\infty}^{2}\big)t}
\big\|g_{0}\big\|
\end{equation}
and
\begin{equation}\label{borne1}
\int_{0}^{t} d\tau \int v^{2}\,
\big|g_{\tau}\big|^{2}\dsdvdx\le \frac{2}{\gamma^{2}}\;
 e^{\big(\gamma
  +4\gamma^{2}\|F\|_{\infty}^{2}\big)t}\;
\big\|g_{0}\big\|^{2}\;.
\end{equation}
Similarly, we get
\begin{equation}\label{borne2}
\partial_{t}\int v^{2}\,\big|g_{t}\big|^{2}\dsdvdx
$$
$$
=4
\int F\,v\,\big|g_{t}\big|^{2}\dsdvdx+
2\gamma 
\int F\,v^{3}\,\big|g_{t}\big|^{2}\dsdvdx
-\gamma^{2}\int v^{4}\,
\big|g_{t}\big|^{2}\dsdvdx
$$
$$
+\gamma \int v^{2}\,\big|g_{t}\big|^{2}\dsdvdx
+2 \int\big|g_{t}\big|^{2}\dsdvdx
-\int v^{2}\,\big|\partial_{v}g_{t}\big|^{2}\dsdvdx
$$ 
$$
\le C(\gamma, \|F\|_{\infty}) \int\big|g_{t}\big|^{2}\dsdvdx
\end{equation}
where $C(\gamma, \|F\|_{\infty})$ is a constant independent of
$g_{t}$. For $t>0$ fixed, we deduce from (\ref{borne1}) that there
exists $\xi(t)\in[0,t[$ such that 
$$
\int v^{2}\,
\big|g_{\xi(t)}\big|^{2}\dsdvdx\le \frac{2}{\gamma^{2}\,t}\;
 e^{\big(\gamma
  +4\gamma^{2}\|F\|_{\infty}^{2}\big)t}\;
\big\|g_{0}\big\|^{2}\;.
$$
Using (\ref{borne2}) and (\ref{borne0}) we get (for any $t>0$) 
$$
\int v^{2}\,
\big|g_{t}\big|^{2}\dsdvdx=\int v^{2}\,
\big|g_{\xi(t)}\big|^{2}\dsdvdx+\int_{\xi(t)}^{t}d\tau \;
\partial_{\tau}\int v^{2}\,
\big|g_{\tau}\big|^{2}\dsdvdx
$$
$$
\le\left( \frac{2}{\gamma^{2}\,t}+
C\,t \right)  e^{\big(\gamma   +4\gamma^{2}\|F\|_{\infty}^{2}\big)t}
\big\|g_{0}\big\|^{2}\;.
$$ 
In other words, for any $t>0$ the image of the unit ball by the
semi-group is equi-integrable at infinity in $v$.  Compactness follows
immediately using hypoelliptic estimates (see \cite{luc}).

From the standard control arguments (see \cite{luc}) applied to the
diffusion (\ref{diff2}), we obtain that the kernel of the semi-group
has dense support. We now observe that integration of a function
against $e^{-\gamma v^{2}}$ is preserved by the semi-group
evolution. This implies by standard arguments that the spectral radius
of the  semi-group is one, the invariant density  is unique, and
exponential mixing  holds (see \cite{luc}). Finally, as in 
equation  (\ref{coercif}), for $g$ in the domain of $\gen'$ we have

$$
\Re \int \bar g\,\big(\gen'_{F}\,g\big)  ds\,\dvdx
$$
$$
\ge 
\left(\frac{\gamma}{2}-2\big\|F\big\|_{\infty}^{2}\right)\int
|g|^{2}\dsdvdx
-\int\left(\frac{\gamma^{2}}{2} v^{2} |g|^{2}
+\frac{1}{2} \big|\partial_{v}g\big|^{2}\right)\dsdvdx\;.
$$
This implies for any $\lambda>0$ 
$$
\int \big|\partial_{v}g-\gamma vg\big|^{2} \dsdvdx\le \frac{1}{\lambda}
\big\|\gen'g\|^{2}_{2}+\big(4\,\lambda+
3\,\gamma+4\,\big\|F\big\|_{\infty}^{2}\big)\|g\|_{2}\;.
$$
In other words, the operator $ F\,\big(\partial_{v}-\gamma v\big)$ is
relatively bounded with respect to $\gen'_{0}$ with relative bound
zero and this implies the analyticity (see \cite{kato}).
\end{proof}

We can now complete the proof of Theorem \ref{principal3}. 

\begin{proof} (of Theorem \ref{principal3}). 

Repeating the argument in the proof  of Proposition \ref{asympt1},
using Proposition \ref{deuxtempsl} we get
$$
\lim_{t\to\infty} \frac{1}{t}\int_{0}^{t}\!\!\!\int\!\!\!\int 
v\,f(\tau,v,x)
d\tau\,dv\,dx=\frac{1}{\gamma T}\int_{0}^{T}\!\!\!\int_{0}^{L}
\!\!\!\int F(\tau,x)
\,\tilde f(\tau,v,x)
d\tau\,dv\,dx\;,
$$

We can now use perturbation theory to compute the right hand side near
$F=0$. For this purpose, it is convenient to fix a $C^{1}$ function
$G$ periodic in space and time and with zero average, and consider
$F=\epsilon\, G$ with $\epsilon$ small. Since $\tilde f$ is analytic
by Proposition \ref{deuxtempsl}, we can write
$$
\tilde f=\tilde f_{0}+\epsilon \tilde f_{1}+\epsilon^{2} \tilde f_{2}+
  \Oun \epsilon^{3}
$$ 
where 
$$
\tilde f_{0}=\frac{1}{L}\sqrt{\frac{\gamma}{\pi}}\;e^{-\gamma v^{2}}
$$
and for $n\ge 1$ the $\tilde f_{n}$ are functions of integral zero,
periodic in time of period $T$, defined recursively by
$$
\partial_{t}\tilde f_{n}+v\partial_{x}\tilde f_{n} -
\partial_{v}\big(\gamma v\tilde f_{n}\big)
-\frac{1}{2}\partial_{v}^{2}\tilde f_{n}=-G(t,x)\partial_{v}\tilde f_{n-1}\;.
$$
We get immediately
$$
\int_{0}^{T}\!\!\!\int_{0}^{L}\!\!\!\int G(\tau,x)
\,\tilde f_{0}(\tau,v,x)
d\tau\,dv\,dx=0\;,
$$
since $\tilde f_{0}$ is independent of $x$ and $t$ and $G$ has average
zero. Therefore we now have to look at the next order perturbation, namely
the second order in $\epsilon$ for  the average velocity. In other
words, we have
\begin{equation}\label{ordre2}
\lim_{t\to\infty} \frac{1}{t}\int_{0}^{t}\!\!\!\int_{0}^{L}\!\!\!\int 
v\,f(\tau,v,x)
d\tau\,dv\,dx=
$$
$$
\frac{\epsilon^{2}}{\gamma T}
\int_{0}^{T}\!\!\!\int_{0}^{L}\!\!\!\int G(\tau,x)
\,\tilde f_{1}(\tau,v,x)
d\tau\,dv\,dx\;+\Oun\epsilon^{3}\;.
\end{equation}

We first have to solve 
$$
\partial_{t}\tilde f_{1}+v\partial_{x}\tilde f_{1} -
\partial_{v}\big(\gamma v\tilde f_{1}\big)
-\frac{1}{2}\partial_{v}^{2}\tilde f_{1}=-G(t,x)\partial_{v}\tilde f_{0}
$$
to get $\tilde f_{1}$. For this purpose, we use Fourier transform in
all variables (recall that $t$ and $x$ are periodic variables).
We will denote by $\hat f_{1,p,q}(k)$ the Fourier transform of
$\tilde f_{1}$ (and similarly for other functions), namely
$$
\tilde f_{1}(t,v,x)=\sum_{p,q}e^{2\pi\,i\, p\,t/T}e^{2\pi\,i\, q\,x/L}
\int e^{i\,k\,v}\hat f_{1,p,q}(k)\,dk\;.
$$
We get 
\begin{equation}\label{eqf1}
\left(\frac{2\pi ip}{T}+\frac{k^{2}}{2}\right)\hat f_{1,p,q}
+\left(\gamma k-\frac{2\pi q}{L}\right)\frac{d}{dk}\hat f_{1,p,q}=-\frac{i\,k\,
\hat G_{p,q}}{2\pi\,L}\,e^{-k^{2}/(4\gamma)}\;.
\end{equation}
We now observe that equation (\ref{ordre2}),  only involves  the
integral of $\tilde f_{1}$  with respect to $v$ since $G$ does not
depend on $v$. Therefore we need only to compute $\hat
f_{1,p,q}(0)$. Let $h_{p,q}(k)$ be
the function
$$
h_{p,q}(k)=e^{k^{2}/(4\gamma)} e^{\pi\, q\,k/(\gamma^{2}L)} 
\,\left|1-\frac{\gamma kL}{2\pi q}
\right|^{2\pi^{2}q^{2}/(\gamma^{3}L^{2})+2\pi ip/(\gamma T)-1}\;.
$$
For $-2\pi |q|/(\gamma L)< k< 2\pi |q|/(\gamma L)$, 
this function is  a solution of 
$$
\left(\frac{2\pi ip}{T}+\frac{k^{2}}{2}\right)h_{p,q}(k)
-\frac{d}{dk}\left[\left(\gamma k-\frac{2\pi
    q}{L}\right)h_{p,q}(k)\right]=0\;. 
$$
For $q>0$, multiplying (\ref{eqf1}) by $h_{p,q}(k)$ 
and integrating over $k$ from $0$ to $2\pi q/(\gamma L)$, we get 
$$
\hat f_{1,p,q}(0)= \hat G_{p,q}\Gamma_{p,q}
$$
where
$$
\Gamma_{p,q}=-
 \frac{i}{4\pi^{2}q}
\int_{0}^{2\pi q/(\gamma L)}
 e^{\pi\, q\,k/(\gamma^{2}L)} 
\,\left|1-\frac{\gamma kL}{2\pi q}
\right|^{2\pi^{2}q^{2}/(\gamma^{3}L^{2})+2\pi ip/(\gamma T)-1}k\,
dk\;.
$$
Note that since $q\neq0$, the integral is convergent.
For $q<0$, one gets a similar result, namely
$$
 \Gamma_{p,q}=
 \frac{i}{4\pi^{2}q}
\int_{2\pi q/(\gamma L)}^{0}
 e^{\pi\, q\,k/(\gamma^{2}L)} 
\,\left|1-\frac{\gamma kL}{2\pi q}
\right|^{2\pi^{2}q^{2}/(\gamma^{3}L^{2})+2\pi ip/(\gamma T)-1}k\,
dk\;,
$$
and it is easy to verify that $\overline
\Gamma_{p,q}=\Gamma_{-p,-q}$. 
We now have
$$
\int_{0}^{T}\!\!\!\int_{0}^{L}\!\!\!\int G(\tau,x)
\,\tilde f_{1}(\tau,v,x)
d\tau\,dv\,dx=\sum_{p,q}\overline{\hat G_{p,q}} \;\Gamma_{p,q}\;  \hat
G_{p,q}\;.
$$
Since $G(t,x)$ is real, we have $\overline{\hat G_{p,q}}= \hat G_{-p,-q}$, and
therefore
$$
\int_{0}^{T}\!\!\!\int_{0}^{L}\!\!\!\int G(\tau,x)
\,\tilde f_{1}(\tau,v,x)
d\tau\,dv\,dx=\sum_{p,q} \big|\hat G_{p,q}\big|^{2}\; \frac{\Gamma_{p,q}+  
\Gamma_{-p,-q}}{2}\;.
$$
We observe that for $q>0$, 
$$
\frac{\Gamma_{p,q}+  
\Gamma_{-p,-q}}{2}=
 \frac{1}{4\pi^{2}q}
\int_{0}^{2\pi q/(\gamma L)}
 e^{\pi\, q\,k/(\gamma^{2}L)} 
\,\left|1-\frac{\gamma kL}{2\pi q}
\right|^{2\pi^{2}q^{2}/(\gamma^{3}L^{2})-1}\,
$$
$$
\times\, 
\sin\big((2\pi p/(\gamma T))\log(1-\gamma kL/(2\pi q))\big)\,k\,
dk\;,
$$
$$
=\frac{q}{\gamma^{2}L}\int_{0}^{1}
e^{2\pi^{2}q^{2}u/(\gamma^{3}L^{2})}
(1-u)^{2\pi^{2}q^{2}/(\gamma^{3}L^{2})-1} 
\sin\big((2\pi p/(\gamma T))\log(1-u)\big)\,u\,
du\;.
$$
Note that this quantity is equal to zero for $p=0$, and it is odd in
$p$. It can be expressed in terms of degenerate hypergeometric
functions.  We now have to
prove that for $p\neq 0$, this quantity is not zero, at least for one
pair of integers  $(p,q)$.  For this purpose, we will consider $q$
large. To alleviate the notation, we consider the asymptotic behavior
for $\alpha>0$ large and $\beta\in\mathbf{R}$ fixed of the integral
$$
J(\alpha)=\int_{0}^{1} e^{\alpha u} (1-u)^{\alpha-1}(1-u)^{i\beta}
 u\,du\;.
$$
Using steepest descent at the critical point $u=0$ (see
\cite{hormander}), one gets
$$
\Im J(\alpha) =
-\frac{\beta}{(2\alpha)^{3/2}\sqrt{\pi}}+\frac{\Oun}{\alpha^{2}}\;. 
$$
We apply this result with $\alpha=2\pi^{2}q^{2}/(\gamma^{3}L^{2})$ and 
$\beta=2\pi p/(\gamma T)$, and conclude that for $q$ large enough, 
$\big(\Gamma_{p,q}+  
\Gamma_{-p,-q}\big)/2\neq 0$ as required.   
\end{proof}

The proof of  Theorem \ref{principal3} is finished  as before.

\subsection{Proof of Theorem \ref{principal4}}

The scheme of the proof is similar to the proofs of the previous
Theorems. We only sketch the argument except for some particular
points. We assume $\gamma>0$, $\nu_{1}>0$ and $\nu_{2}>0$.
 One starts by reducing the problem to a periodic boundary condition
 in $x$. The key result is the analog of Proposition \ref{et2}.

\begin{proposition}\label{et4}
The semi-group defined by (\ref{kramers2}) is compact in
$L^{2}\big(e^{\gamma v^{2}}dv\,dx\big)$ (the $x$ variable being on the
circle of length $L$).  It is positivity preserving positivity
improving on the diagonal. Its peripheral spectrum is the simple
eigenvalue one.  The corresponding eigenvector $(\tilde f_{1}, \tilde
f_{2})$ can be chosen positive with dense domain and normalised
($\tilde f_{1}+\tilde f_{2}$ of integral one). The functions $\tilde
f_{1}$ and $\tilde f_{2}$ depend analytically on $F_{1}$, and $F_{2}$.
\end{proposition}

\begin{proof}
The proof is very similar to the proof of Proposition \ref{et2} and we
only sketch the details for the different points. 

\end{proof}

It follows as before from the evolution equation that 
$$
I\big(\gamma,F_{1},F_{2},\nu_{1},\nu_{2}\big)=
\int\big(\tilde f_{1} F_{1}+\tilde f_{2} F_{2}\big)\dvdx\;.
$$
For fixed  $\gamma>0$, $\nu_{1}>0$ and $\nu_{2}>0$, this quantity is
real analytic in  $(F_{1},F_{2})$, and to check that it is non trivial
we investigate its behaviour near the origin. For this purpose, we set 
$(F_{1},F_{2})=\epsilon(G_{1},G_{2})$ with $G_{1}$ and $G_{2}$ two
$C^{1}$ functions, periodic of period $L$ and with zero average. We
now develop $I\big(\gamma,\epsilon G_{1},\epsilon
G_{2},\nu_{1},\nu_{2}\big)$ in series of $\epsilon$.
As for  model (\ref{modele2}), the terms of order $\epsilon$
and of order $\epsilon^{2}$ vanish. One then has to find
 a pair $(G_{1},G_{2})$
such that the term of order $\epsilon^{3}$ does not vanish.
The computations are  rather tedious  and will be detailed in the
appendix.

\appendix
\section{Proof of Theorem \ref{principal4}.}

We define the vectors 
$$
\vec f=\left(\begin{array}{c}
\tilde f_{1}\\
\tilde f_{2}\\
\end{array}\right)\quad\textrm{and}\quad 
\vec F=\left(\begin{array}{c}
 F_{1}\\
 F_{2}\\
\end{array}\right)\;.
$$
We then have
$$
I\big(\gamma,F_{1},F_{2},\nu_{1},\nu_{2}\big)=
\iint \vec F\bullet \vec f\;\dvdx\;.
$$
The  equation for the stationary solution of (\ref{kramers2}) can be
written 
\begin{equation}\label{statio1}
\frac{1}{2}\partial_{v}^{2}\vec f-v\partial_{x}\vec f
+\partial_{v}\big[\gamma v\vec f\big]+M\vec f=B(x)\;\partial_{v}\vec f
\end{equation}
where as before 
$$
M=\left(\begin{array}{cc}
-\nu_{1}&\nu_{2}\\
\nu_{1}&-\nu_{2}\\
\end{array}\right)
$$
and
$$
B(x)=\left(\begin{array}{cc}
F_{1}(x)&0\\
0&F_{2}(x)\\
\end{array}\right)\;.
$$
Without loss of generality, we can rescale $v$ and assume $\gamma=1$,
replacing $\nu_{1}$ and  $\nu_{2}$ by $\nu_{1}/\gamma$ and
$\nu_{2}/\gamma$ respectively, $L$ by $L \gamma^{-3/2}$,  $F_{1}$ and  $F_{2}$
by $F_{1}/\sqrt \gamma$ and $F_{2}/\sqrt \gamma$ respectively.
 
For reasons that will become obvious later on, it is more convenient
to consider instead of $\vec f$ the vector 
$$
\vec \psi= e^{v^{2}/2}\vec f\;.
$$
Let $\mathscr{L}$ be the operator defined by 
$$
\mathscr{L}h(v)=\partial_{v}^{2}h(v)-v^{2}h(v)+h\;,
$$
then the equation (\ref{statio1}) for the stationary solution can be
written
\begin{equation}\label{statio2}
\mathscr{L}\vec \psi-2 v\partial_{x} \vec \psi+ 2M  \vec
\psi
=2B(x)\;\big(\partial_{v}\vec \psi -v \vec \psi\big)\;,
\end{equation}
and we have
\begin{equation}\label{leI}
I\big(\gamma,F_{1},F_{2},\nu_{1},\nu_{2}\big)=
\iint \vec F\bullet \vec \psi\;e^{-v^{2}/2}\;\dvdx\;,
\end{equation}
with $\vec \psi$ normalised by
\begin{equation}\label{norme}
\iint \big(\psi_{1}+\psi_{2}\big) \;e^{-v^{2}/2}\;\dvdx=1\;.
\end{equation}
For $\vec F=\epsilon\,\vec G$ we have the expansion
\begin{equation}\label{psidev}
 \vec\psi= \vec\psi^{\;0}+\epsilon  \vec\psi^{\;1} +\epsilon^{2}
 \vec\psi^{\;2} +\mathscr{O}\big(\epsilon^{3}\big)
\end{equation}
where
$$
\vec\psi^{\;0}(v,x)=\frac{1}{L\big(\nu_{1}+\nu_{2}\big)\sqrt\pi}
\left(\begin{array}{c}
\nu_{2}\\
\nu_{1}\\
\end{array}\right)e^{-v^{2}/2}\;,
$$
and the vectors $\vec\psi^{\;j}$ are obtained recursively by solving
$$
\mathscr{L}\vec \psi^{\;j+1}-2 v\partial_{x} \vec \psi^{\;j+1}+ 2M  \vec
\psi^{\;j+1}
=2C(x)\;\big(\partial_{v}\vec \psi^{\;j} -v \vec \psi^{\;j}\big)\;,
$$
where 
$$
C(x)=\left(\begin{array}{cc}
G_{1}(x)&0\\
0&G_{2}(x)\\
\end{array}\right)\;.
$$
Note that $\vec \psi^{\;0}$ satisfies
$$
\mathscr{L}\vec \psi^{\;0}-2 v\partial_{x} \vec \psi^{\;0}+ 2M  \vec
\psi^{\;0}=\vec0
$$
and the normalisation condition (\ref{norme}), while for $j\ge1$, 
$\vec \psi^{\;j}$ should satisfy the normalisation condition
\begin{equation}\label{normal}
\iint \big(\psi_{1}^{j}+\psi_{2}^{j}\big) \;e^{-v^{2}/2}\;\dvdx=0\;.
\end{equation}
 Corresponding to the
expansion (\ref{psidev}), we have the expansion
$$
I\big(1,\epsilon G_{1},\epsilon G_{2},\nu_{1},\nu_{2}\big)=
\sum_{n=1}^{3} \epsilon^{n}
I_{n}\big(G_{1},G_{2},\nu_{1},\nu_{2}\big)\
+\mathscr{O}\big(\epsilon^{4}\big)\;,
$$
where 
$$
I_{n}\big(G_{1},G_{2},\nu_{1},\nu_{2}\big)
=\iint \vec G\bullet \vec \psi^{\;n-1}\;e^{-v^{2}/2}\;\dvdx\;.
$$
We have  immediately
$I_{1}\big(G_{1},G_{2},\nu_{1},\nu_{2}\big)=0$ since the averages of
$G_{1}$ and $G_{2}$ over $x$ vanish and $\vec \psi^{\;0}$ does not
depend on $x$.
In order to compute $I_{1}$, we need to compute $\vec\psi^{\;1}$ which
solves
$$
\mathscr{L}\vec \psi^{\;1}-2 v\partial_{x} \vec \psi^{\;1}+ 2M  \vec
\psi^{\;1}
=v\, e^{-v^{2}/2}\vec H(x)
$$
where 
$$
\vec H(x)=-\frac{4}{L\big(\nu_{1}+\nu_{2}\big)\sqrt\pi}
\left(\begin{array}{c}
\nu_{2}G_{1}(x)\\
\nu_{1}G_{2}(x)\\
\end{array}\right)\;.
$$
Since $\vec H$ is periodic of period $L$, we can decompose this
equation in Fourier series and get
$$
\mathscr{L}\vec \psi^{\;1}(p)-\frac{4\pi \,i\,p\,v}{L}
 \vec \psi^{\;1}(p)+ 2M  \vec
\psi^{\;1}(p)
=v\, e^{-v^{2}/2}\vec H(p)\;.
$$ 
We now observe that 
$$
\mathscr{L}-\frac{4\pi \,i\,p\,v}{L}+2M=e^{2\pi \,i\,p\,\partial_{v}/L}
\left(\mathscr{L} -\frac{4 \pi^{2}p^{2}}{L^{2}}+2M\right)
e^{-2\pi \,i\,p\,\partial_{v}/L}\;.
$$
This implies 
\begin{equation}\label{lepsi1}
 \vec
\psi^{\;1}(p) =e^{2\pi \,i\,p\,\partial_{v}/L}
\left(\mathscr{L} -\frac{4 \pi^{2}p^{2}}{L^{2}}+2M\right)^{-1}
e^{-2\pi \,i\,p\,\partial_{v}/L}\left(v\, e^{-v^{2}/2}\vec H(p)\right)\;.
\end{equation}
The spectrum of the self adjoint operator $\mathscr{L}$  in
$L^{2}(dv)$ is well known (see for example \cite{as} or \cite{gr}).
 The eigenvalues are the numbers $-2n$ with
$n\in\mathbf{N}^{*}$, and the corresponding normalised eigenvectors
are the functions
$$
e_{n}(v)=\frac{2^{-n/2}}{\pi^{1/4}\,(n!)^{1/2}}\;
H_{n}(v)\;e^{-v^{2}/2}
=\frac{(-1)^{n}\;2^{-n/2}}{\pi^{1/4}\,(n!)^{1/2}}\;
\;e^{v^{2}/2}\;\frac{d^{n}}{dv^{n}}e^{-v^{2}}\;, 
$$ 
where the $H_{n}$ are the Hermite polynomials. We have
$$
e^{-2\pi \,i\,p\,\partial_{v}/L}\left(v\, e^{-v^{2}/2}\right)=
\left(v-\frac{2\pi \,i\,p}{L}\right)e^{-v^{2}/2}e^{2\pi \,i\,p\,v/L} 
e^{2\,\pi^{2}\,p^{2}/L^{2}}\;,
$$
and we decompose this function on the Hermite basis $(e_{n})$.
It follows easily using integration by parts that
$$
\int e_{n}(v) \left(v-\frac{2\pi \,i\,p}{L}\right)e^{-v^{2}/2}e^{2\pi
  \,i\,p\,v/L}  
e^{2\,\pi^{2}\,p^{2}/L^{2}}dv
$$
$$
= (-1)^{n}\;e^{2\,\pi^{2}\,p^{2}/L^{2}}
\frac{2^{-n/2}}{\pi^{1/4}\,(n!)^{1/2}} \int 
\left(v-\frac{2\pi \,i\,p}{L}\right)e^{2\pi
  \,i\,p\,v/L}\;
\frac{d^{n}}{dv^{n}}e^{-v^{2}} dv
$$
$$
= e^{2\,\pi^{2}\,p^{2}/L^{2}}
\frac{2^{-n/2}}{\pi^{1/4}\,(n!)^{1/2}}\; \times
$$
$$
\int 
\left[v\left(\frac{2\pi \,i\,p}{L}\right)^{n}-\left(\frac{2\pi
    \,i\,p}{L}\right)^{n+1}+n\left(\frac{2\pi
    \,i\,p}{L}\right)^{n-1}\right]
e^{-v^{2}} e^{2\pi
  \,i\,p\,v/L}\;dv
$$
$$
=e^{\,\pi^{2}\,p^{2}/L^{2}}
\frac{2^{-n/2}\pi^{1/4}}{(n!)^{1/2}} \;\big(n+2\pi^{2}p^{2}/L^{2}\big)\;
\left(\frac{2\pi
    \,i\,p}{L}\right)^{n-1}\;.
$$
This formula holds for any $p\neq0$ and any $n\ge0$. For $p=0$, the
formula holds  for any $n\ge1$ (with the convention $0^{0}=1$).
 Finally, for $n=0$ and $p=0$ the
result is zero. 
We now have after some simple linear algebra
$$
\left(\mathscr{L} -\frac{4
  \pi^{2}p^{2}}{L^{2}}+2M\right)^{-1}\left(e_{n}\,\vec H(p)\right)(v)=
e_{n}(v)\left(-2n -\frac{4
  \pi^{2}p^{2}}{L^{2}}+2M\right)^{-1}\vec H(p)
$$
$$
=\frac{e_{n}(v)}{2 \big(n +2
  \pi^{2}p^{2}/L^{2}\big)\big(n +\nu_{1}+\nu_{2}+2
  \pi^{2}p^{2}/L^{2}\big)} 
$$
$$
\times \left(\begin{array}{cc}
-\nu_{2}-n -2\pi^{2}p^{2}/L^{2}&
-\nu_{2}\\
-\nu_{1}&
-\nu_{1}-n -2   \pi^{2}p^{2}/L^{2}\\
\end{array}\right)\vec H(p)\;.
$$
We can write using Fourier series
$$
I_{2}\big(G_{1},G_{2},\nu_{1},\nu_{2}\big)
=\sum_{p} \overline
{\vec G(p)}\bullet \int  \vec \psi^{\;1}(v,p)\;e^{-v^{2}/2}\;dv
$$
$$
=\sum_{p}\sum_{n} \overline
{\vec G(p)}\bullet
\left(\begin{array}{cc}
-\nu_{2}-n -2\pi^{2}p^{2}/L^{2}&
-\nu_{2}\\
-\nu_{1}&
-\nu_{1}-n -2   \pi^{2}p^{2}/L^{2}\\
\end{array}\right)
\vec H(p)
$$
$$
\times\; \frac{1}{2 
\big(n +\nu_{1}+\nu_{2}+2
  \pi^{2}p^{2}/L^{2}\big)} \;
e^{\,\pi^{2}\,p^{2}/L^{2}}
\frac{2^{-n/2}\pi^{1/4}}{(n!)^{1/2}} 
\left(\frac{2\pi
    \,i\,p}{L}\right)^{n-1}
$$
$$
\times\; \int \left(e^{2\pi \,i\,p\,\partial_{v}/L}e_{n}\right)(v)
\;e^{-v^{2}/2}\;dv\;.
$$
Since the operator $i\partial_{v}$ is self adjoint, we get
$$
\int  
e^{-v^{2}/2} \left(e^{2\pi \,i\,p\,\partial_{v}/L}e_{n}\right)(v)
\;dv=
\int  
\overline{e^{2\pi \,i\,p\,\partial_{v}/L} e^{-v^{2}/2} }\;e_{n}(v)
\;dv
$$
$$
=\int  
e_{n}(v)
\;e^{-\big(v-2\pi \,i\,p/L\big)^{2}/2}\;dv
=e^{2\,\pi^{2}\,p^{2}/L^{2}}\int  
e_{n}(v)
\;e^{-v^{2}/2}\;
e^{2\,i\,v\,\pi\,p\,v/L}\;dv
\;.
$$
By a computation similar to the  above one, we get
$$
\int  
e^{2\pi \,i\,p\,v/L}e_{n}(v)
\;e^{-v^{2}/2}\;dv=
e^{\,\pi^{2}\,p^{2}/L^{2}}
\frac{2^{-n/2}\pi^{1/4}}{(n!)^{1/2}}
\left(\frac{2\pi
    \,i\,p}{L}\right)^{n}\;.
$$
After some simple algebra, we obtain 
$$
I_{2}\big(G_{1},G_{2},2\nu_{1},2\nu_{2}\big)
$$
$$
=\frac{2}{L(\nu_{1}+\nu_{2})}\sum_{n,\,p,\, n+p^{2}>0}
e^{2\,\pi^{2}\,p^{2}/L^{2}}
\frac{2^{-n}}{n!}
\left(\frac{2\pi
    \,i\,p}{L}\right)^{2n-1}
\frac{R(p)}{\big(n +\nu_{1}+\nu_{2}+2
  \pi^{2}p^{2}/L^{2}\big)}
$$
where 
$$
R(p)=\overline{G_{1}(p)}G_{1}(p)\nu_{2}\big(\nu_{2}+n+2
  \pi^{2}p^{2}/L^{2}\big)
+\overline{G_{2}(p)}G_{2}(p)\nu_{1}\big(\nu_{1}+n+2
  \pi^{2}p^{2}/L^{2}\big)
$$
$$
+\nu_{1}\nu_{2}\big(\overline{G_{1}(p)}G_{2}(p)
+\overline{G_{2}(p)}G_{1}(p)\big)\;.
$$

Since $G_{1}$ and $G_{2}$ are real, we have $\overline{G_{1}(p)}=
G_{1}(-p)$ and $\overline{G_{2}(p)}= G_{2}(-p)$. This implies
immediately that $R(p)$ is even in $p$ and therefore 
$I_{2}\big(G_{1},G_{2},\nu_{1},\nu_{2}\big)=0$.  A similar computation
(involving only $p=0$) shows that the normalisation condition
(\ref{normal}) holds.
We now need to compute
$I_{3}\big(G_{1},G_{2},\nu_{1},\nu_{2}\big)$, and we recall that it is
given by
\begin{equation}\label{lei3ini}
I_{3}\big(G_{1},G_{2},\nu_{1},\nu_{2}\big)=
\sum_{p} \overline
{\vec G(p)}\bullet \int  \vec \psi^{\;2}(v,p)\;e^{-v^{2}/2}\;dv
\end{equation}
where
$$
\vec \psi^{\;2}(v,p)=e^{2\pi \,i\,p\,\partial_{v}/L}
\left(\mathscr{L} -\frac{4 \pi^{2}p^{2}}{L^{2}}+2M\right)^{-1}
e^{-2\pi \,i\,p\,\partial_{v}/L}\left(\vec J(v,p)\right)\;.
$$
and where $\vec J(v,p)$ is the Fourier series of the vector 
$$
\vec J(v,x)=
2 C(x)\left(\partial_{v}\vec\psi^{\;1}(v,x)-v\vec\psi^{\;1}(v,x)\right)\;.
$$
As above, using that the operator $i\partial_{v}$ is self adjoint, we
get
$$
I_{3}\big(G_{1},G_{2},2\nu_{1},2\nu_{2}\big)=
\sum_{p}
\int\overline{\left(\mathscr{L} -\frac{4
    \pi^{2}p^{2}}{L^{2}}+2M^{t}\right)^{-1}
e^{2\pi \,i\,p\,\partial_{v}/L}\left(e^{-v^{2}/2}\vec G(p)\right)}
$$
$$
\bullet\; e^{-2\pi \,i\,p\,\partial_{v}/L}\left(\vec J(v,p)\right)\;dv\;.
$$
We can now compute everything in terms of series of Hermite
coefficients. 
Denoting   by $C(p)$ the Fourier coefficients of the matrix
$C(x)$ we  get 
$$
e^{-2\pi \,i\,p\,\partial_{v}/L}\vec J(v,p)=2\sum_{q}C(p-q) e^{-2\pi
  \,i\,p\,\partial_{v}/L} 
\left(\partial_{v}\vec\psi^{\;1}(v,q)-v\vec\psi^{\;1}(v,q)\right)
$$
$$
=2\sum_{q}C(p-q)\; e^{-2\pi \,i\,p\,\partial_{v}/L}
$$
$$
\big(\partial_{v}-v\big)e^{2\pi \,i\,q\,\partial_{v}/L}
\left(\mathscr{L} -\frac{4 \pi^{2}q^{2}}{L^{2}}+2M\right)^{-1}
e^{-2\pi \,i\,q\,\partial_{v}/L}\left(v\, e^{-v^{2}/2}\vec H(q)\right)
$$
$$
=2\sum_{q}C(p-q) e^{2\pi \,i\,(p-q)\,\partial_{v}/L}
$$
$$
\big(\partial_{v}-v+2\pi \,i\,q/L\big)
\left(\mathscr{L} -\frac{4 \pi^{2}q^{2}}{L^{2}}+2M\right)^{-1}
e^{-2\pi \,i\,q\,\partial_{v}/L}\left(v\, e^{-v^{2}/2}\vec H(q)\right).
$$
We have already computed (for $n\ge 0$ if $p\neq0$ and for $p=0$ if
$n>0$) 
\begin{equation}\label{parte1}
\left(\mathscr{L} -\frac{4 \pi^{2}q^{2}}{L^{2}}+2M\right)^{-1}
e^{-2\pi \,i\,q\,\partial_{v}/L}\left(v\, e^{-v^{2}/2}\vec H(q)\right)
$$
$$
=\sum_{n}
e^{\,\pi^{2}\,q^{2}/L^{2}}
\frac{2^{-n/2}\pi^{1/4}}{(n!)^{1/2}}
\left(\frac{2\pi
    \,i\,q}{L}\right)^{n-1}
 \frac{e_{n}(v)}{2 
\big(n +\nu_{1}+\nu_{2}+2
  \pi^{2}q^{2}/L^{2}\big)}
$$
$$
\left(\begin{array}{cc}
-\nu_{2}-n -2\pi^{2}q^{2}/L^{2}&
-\nu_{2}\\
-\nu_{1}&
-\nu_{1}-n -2   \pi^{2}q^{2}/L^{2}\\
\end{array}\right)
\vec H(q)
\end{equation}
By a similar computation left to the reader, we get
\begin{equation}\label{parte2}
\left(\mathscr{L} -\frac{4
    \pi^{2}p^{2}}{L^{2}}+2M^{t}\right)^{-1}
e^{2\pi \,i\,p\,\partial_{v}/L}\left(e^{-v^{2}/2}\vec G(p)\right)
$$
$$
=\sum_{m}
e^{\,\pi^{2}\,p^{2}/L^{2}}
\frac{2^{-m/2}\pi^{1/4}}{(m!)^{1/2}}
\left(-\frac{2\pi
    \,i\,p}{L}\right)^{m}
$$
$$
\times \frac{e_{m}(v)}{2 \big(m +2
  \pi^{2}p^{2}/L^{2}\big)\big(m +\nu_{1}+\nu_{2}+2
  \pi^{2}p^{2}/L^{2}\big)}
$$
$$
\left(\begin{array}{cc}
-\nu_{2}-m -2\pi^{2}p^{2}/L^{2}&
-\nu_{1}\\
-\nu_{2}&
-\nu_{1}-m -2   \pi^{2}p^{2}/L^{2}\\
\end{array}\right)
\vec G(p)
\end{equation}

It finally remains to compute
$$
\int e_{m}(v) e^{-2\,\pi\, i\, (p-q)\,\partial_{v}/L}
\big(\partial_{v}-v+2\pi \,i\,q/L\big) e_{n}(v)dv
$$
$$
=-\sqrt2\,\sqrt{n+1}
\int e_{m}(v) e_{n+1}\big(v-2\,\pi\, i\, (p-q)/L\big)\,dv
$$
$$
+\frac{2\pi \,i\,q}{L} \int e_{m}(v) 
e_{n}\big(v-2\,\pi\, i\, (p-q)/L\big)\,dv\;,
$$
since
$$
\left(\frac{d}{dv}-v\right)e_{n}(v)=-\sqrt2\,\sqrt{n+1}\;e_{n+1}(v)\;.
$$
We now compute the quantity 
$$
\gamma_{n,m}(r) =\int e_{m}(v) 
e_{n}\big(v-2\,\pi\, i\, r/L\big)\,dv
$$
$$
=
\frac{2^{-(n+m)/2}e^{2\,\pi^{2}\,r^{2}/L^{2}}}
{\sqrt\pi\;(n!\,m!)^{1/2}}
\int
e^{-v^{2}}\;e^{2\,\pi\, i\, r\,v/L}
 H_{m}(v)\;H_{n}\big(v-2\,\pi\, i\, r/L\big)\;dv
$$
$$
=\frac{2^{-(n+m)/2}e^{\pi^{2}\,r^{2}/L^{2}}}
{\sqrt\pi\;(n!\,m!)^{1/2}}
\int
e^{-\big(v-\pi\, i\, r/L\big) ^{2}}\;
 H_{m}(v)\;H_{n}\big(v-2\,\pi\, i\, r/L\big)\;dv
$$
$$
=\frac{2^{-(n+m)/2}e^{\pi^{2}\,r^{2}/L^{2}}}
{\sqrt\pi\;(n!\,m!)^{1/2}}
\int e^{-v^{2}}
 H_{m}\big(v+\,\pi\, i\, r/L\big)\;H_{n}\big(v-\,\pi\, i\, 
r/L\big)\;dv\;.
$$
When $m\le n$, we obtain 
$$ 
\gamma_{n,m}(r)
=\frac{(m!)^{1/2}2^{(n-m)/2}e^{\pi^{2}\,r^{2}/L^{2}}}
{(n!)^{1/2}}\left(-\,\frac{i\,\pi\, r}{L}\right)^{n-m}
L_{m}^{n-m}\left(-\,\frac{2\pi^{2}\, r^{2}}{L^{2}}\right)\;,
$$
where $L_{r}^{s}$ denotes a Laguerre polynomial (see \cite{gr},
formula \textbf{7.377}).  Similarly, when $m\ge n$ we get
$$
\gamma_{n,m}(r) =\frac{(n!)^{1/2}2^{(m-n)/2}e^{\pi^{2}\,r^{2}/L^{2}}}
{(m!)^{1/2}}
\left(\,\frac{i\,\pi\, r}{L}\right)^{m-n}
L_{n}^{m-n}\left(-\,\frac{2\pi^{2}\, r^{2}}{L^{2}}\right)\;.
$$
We can now use this result together with equations (\ref{parte1}) and
(\ref{parte2}) into the expression (\ref{lei3ini}). We get 
$$
I_{3}\big(G_{1},G_{2},\nu_{1},\nu_{2}\big)=
$$
$$
-
\frac{1}{L\big(\nu_{1}+\nu_{2}\big)}
\sum_{n,m,p,q}
\left(\frac{2\pi\,i\, q}{L}\gamma_{n,m}(p-q)-
\sqrt2\,\sqrt{n+1}\gamma_{n+1,m}(p-q)\right)
$$
$$
e^{\,\pi^{2}\,q^{2}/L^{2}}
\frac{2^{-n/2}\pi^{1/4}}{(n!)^{1/2}}
\left(\frac{2\pi
    \,i\,q}{L}\right)^{n-1}
e^{\,\pi^{2}\,p^{2}/L^{2}}
\frac{2^{-m/2}\pi^{1/4}}{(m!)^{1/2}}
\left(\frac{2\pi
    \,i\,p}{L}\right)^{m}
$$
$$
\times \frac{ S(p,q)}{ 
\big(m +2
  \pi^{2}p^{2}/L^{2}\big)
\big(m +\nu_{1}+\nu_{2}+2
  \pi^{2}p^{2}/L^{2}\big)
\big(n +\nu_{1}+\nu_{2}+2
  \pi^{2}q^{2}/L^{2}\big)} 
$$
where 
$$
S(p,q)=\left(\left(\begin{array}{cc}
\nu_{2}+m +2\pi^{2}p^{2}/L^{2}&
\nu_{1}\\
\nu_{2}&
\nu_{1}+m +2   \pi^{2}p^{2}/L^{2}\\
\end{array}\right)
\overline{\vec G(p)}\right)\bullet
$$
$$
\left(\left(\begin{array}{cc}
G_{1}(p-q)&
0\\
0&
G_{2}(p-q)\\
\end{array}\right)
\;
\left(\begin{array}{cc}
\nu_{2}+n +2\pi^{2}q^{2}/L^{2}&
\nu_{2}\\
\nu_{1}&
\nu_{1}+n +2   \pi^{2}q^{2}/L^{2}\\
\end{array}\right)\right.
$$
$$
\times \;\left.\left(\begin{array}{c}
\nu_{2}G_{1}(q)\\
\nu_{1}G_{2}(q)\\
\end{array}\right)\right)\;.
$$
A numerical simulation using for $G1$ and $G2$ real Fourier
polynomials of order two and random coefficients give a nonzero result
for $I_{3}\big(G_{1},G_{2},1,2\big)$ and $L=10$.


\end{document}